\documentclass[a4paper,reqno]{amsart}

\usepackage{amssymb}
\usepackage{amstext}
\usepackage{amsmath}
\usepackage{amsthm}
\usepackage{amsfonts}
\usepackage{enumerate}
\usepackage{graphicx}
\usepackage{latexsym}
\usepackage{mathrsfs}
\usepackage{mathtools}
\usepackage{paralist}
\usepackage{xcolor}
\usepackage[colorlinks=true,linkcolor=blue!55!black,citecolor=blue!55!black,urlcolor=blue!55!black]{hyperref}

\hypersetup{
  pdftitle={Self-Orthogonal Tau-Tilting Modules and Tilting Modules III: Finitistic Dimension and Radical Cube Zero Algebras},
  pdfauthor={Xiaojin Zhang and Panyue Zhou},
  pdfsubject={Self-Orthogonal $\tau$-Tilting Modules and Tilting Modules},
  pdfkeywords={tau-tilting module, self-orthogonal module, tilting module, finitistic dimension, annihilator, radical cube zero}
}

\newtheorem{theorem}{Theorem}[section]
\newtheorem{corollary}[theorem]{Corollary}
\newtheorem{lemma}[theorem]{Lemma}
\newtheorem{proposition}[theorem]{Proposition}

\theoremstyle{definition}
\newtheorem{definition}[theorem]{Definition}
\newtheorem{remark}[theorem]{Remark}
\newtheorem{example}[theorem]{Example}
\newtheorem*{theorema}{Theorem A}
\newtheorem*{theoremb}{Theorem B}
\newtheorem*{theoremc}{Theorem C}

\newcommand{\Ext}{\operatorname{Ext}\nolimits}
\newcommand{\Hom}{\operatorname{Hom}\nolimits}
\newcommand{\End}{\operatorname{End}\nolimits}
\newcommand{\Ann}{\operatorname{Ann}\nolimits}
\newcommand{\rad}{\operatorname{rad}\nolimits}
\newcommand{\add}{\mathsf{add}\hspace{.01in}}
\newcommand{\Fac}{\mathsf{Fac}\hspace{.01in}}

\newcommand{\modA}{\mathsf{mod}\hspace{.01in}A}
\newcommand{\findim}{\operatorname{findim}\nolimits}
\newcommand{\pd}{\operatorname{pd}\nolimits}

\numberwithin{equation}{section}

\begin{document}

\title[Self-Orthogonal $\tau$-Tilting Modules and Tilting Modules]{Self-Orthogonal $\tau$-Tilting Modules and Tilting Modules III: Finitistic Dimension}
\thanks{2020 Mathematics Subject Classification: 16G10, 16E10, 16E05.}
\thanks{Keywords: $\tau$-tilting module; self-orthogonal module; tilting module; finitistic dimension; annihilator; radical cube zero.}

\author{Xiaojin Zhang}
\address{X. Zhang: School of Mathematics and Statistics, Jiangsu Normal University, 221116 Xuzhou, P. R. China}
\email{xjzhang@jsnu.edu.cn}

\author{Panyue Zhou}
\address{P. Zhou: School of Mathematics and Statistics, Changsha University of Science and Technology, 410114 Changsha, Hunan, P. R. China}
\email{panyuezhou@163.com}

\begin{abstract}
Let $A$ be an Artin algebra and let $T$ be a $\tau$-tilting right $A$-module. Put $C=A/\Ann_A(T)$. We prove that finite little finitistic dimension of $C^{\rm op}$, together with a finite range of self-extension vanishing for $T$, forces $T$ to be a classical $1$-tilting module. We also show that if the finitistic conjecture holds for all Artin algebras then the self-orthogonal $\tau$-tilting conjecture holds for all Artin algebras. Moreover, a counterexample to the self-orthogonal $\tau$-tilting conjecture will give a counterexample to the finitistic dimension conjecture. As an application, the self-orthogonal $\tau$-tilting conjecture holds for Artin algebras with radical cube zero.
\end{abstract}

\maketitle

\section{Introduction}

Let $A$ be an Artin algebra. Classical tilting modules play a central role in the representation theory of Artin algebras. We refer to Happel and Ringel \cite{HappelRingel} for the classical theory. Adachi, Iyama and Reiten \cite{AIR} introduced $\tau$-tilting theory and showed that it extends classical tilting theory in a natural way. In particular, classical $1$-tilting right $A$-modules are precisely faithful $\tau$-tilting right $A$-modules \cite[Proposition~2.2]{AIR}.

A right $A$-module $T$ is called self-orthogonal if
$$
\Ext_A^i(T,T)=0\quad(i\geq1).
$$
Zhang \cite{Zhang} asked whether every self-orthogonal $\tau$-tilting right $A$-module is a classical $1$-tilting right $A$-module. Following Chen, Li, Zhang and Zhao \cite{CLZZ}, we refer to this statement as the self-orthogonal $\tau$-tilting conjecture. The problem is connected with classical conjectures on self-orthogonal modules. In the self-injective case, Zhang \cite[Section~3]{Zhang} explains its relation to the Auslander-Reiten conjecure \cite{AuslanderReiten} as well as Tachikawa's second conjecture \cite{Tachikawa}. Enomoto \cite{Enomoto} places the problem in a broader framework of maximal self-orthogonal modules. Very recently, Chang and Tang \cite{ChangTang} gave counterexamples to the Self-orthogonal faithful conjecture of Chen, Li, Zhang and Zhao and to Enomoto's Self-orthogonal Wakamatsu-tilting Conjecture. These counterexamples do not settle the self-orthogonal $\tau$-tilting conjecture considered here. The conjecture remains open in general.

Several positive cases are known. Zhang \cite[Theorem~3.4]{Zhang} proved that a self-orthogonal $\tau$-tilting right $A$-module is $1$-tilting if it has finite projective dimension. Chen, Li, Zhang and Zhao \cite[Proposition~5.3]{CLZZ} obtained positive answers for local algebras, syzygy-finite Gorenstein algebras, algebras of finite representation type, and minimal representation-infinite algebras. Marczinzik \cite[Theorem~1.1]{Marczinzik} proved finite projective dimension for self-orthogonal modules over Iwanaga--Gorenstein rings whose indecomposable non-projective maximal Cohen--Macaulay modules are periodic. In particular, his result covers monomial Iwanaga--Gorenstein algebras. Finite-dimensional gentle algebras belong to this class because they are monomial and are Iwanaga--Gorenstein by Geiss and Reiten \cite{GeissReiten}. More recently, Zhang \cite{ZhangII} developed an annihilator separation method and proved the conjecture for Artin algebras with radical square zero.

The recent approaches also reveal two homological objects attached to $T$. One is the endomorphism algebra
$$
B=\End_A(T)^{\rm op}.
$$
Chen, Li, Zhang and Zhao \cite[Theorem~5.4]{CLZZ} use the global delooping level of $B$, while Zhang \cite{ZhangII} obtains a finitistic dimension obstruction involving $B$. The other object is the annihilator quotient
$$
I=\Ann_A(T),\quad C=A/I.
$$
By Adachi, Iyama and Reiten \cite[Proposition~2.2]{AIR}, the right $C$-module $T$ is a classical $1$-tilting module. The original problem is therefore equivalent to proving that $I=0$.

Our main result uses the little finitistic dimension of $C^{\rm op}$. For an Artin algebra $\Gamma$, we write $\findim(\Gamma)$ for the supremum of the finite projective dimensions of finitely generated right $\Gamma$-modules. The finitistic dimension problem goes back to Bass \cite{Bass}. We use the opposite algebra because $\findim(C^{\rm op})$ is the corresponding invariant for finitely generated left $C$-modules.

\begin{theorema}
Let $T$ be a $\tau$-tilting right $A$-module and put $C=A/\Ann_A(T)$. Assume that
$$
d=\findim(C^{\rm op})<\infty.
$$
If
$$
\Ext_A^i(T,T)=0\quad(1\leq i\leq d+2),
$$
then $T$ is a classical $1$-tilting right $A$-module.
\end{theorema}

The theorem has a positive interpretation. It is enough to control the little finitistic dimension of the annihilator quotient. No bound on the finitistic dimension of $A$ is required. Moreover, full self-orthogonality is stronger than necessary. Once $d$ is known, the first $d+2$ self-extension groups already decide faithfulness. The hypothesis is therefore different from the finite projective dimension condition in \cite[Theorem~3.4]{Zhang}. Our condition is imposed on $C^{\rm op}$ rather than on $\pd_A T$. This makes the criterion useful for classes of quotient algebras whose finitistic dimensions are known to be finite.

The theorem also places a strong restriction on any possible counterexample. Rather than producing such a counterexample, it shows that one can exist only when the corresponding annihilator quotient has infinite finitistic dimension.

\begin{theoremb}
If a self-orthogonal $\tau$-tilting right $A$-module $T$ is not a classical $1$-tilting right $A$-module, then
$$
\findim\bigl((A/\Ann_A(T))^{\rm op}\bigr)=\infty.
$$
In fact, $A/\Ann_A(T)$ has finitely generated left modules of every prescribed positive finite projective dimension.
\end{theoremb}

Thus a counterexample to the self-orthogonal $\tau$-tilting conjecture would yield an explicit Artin algebra with infinite little finitistic dimension, which is a counterexample to the finitistic dimension conjecture. At the level of conjectures, it is known that the finitistic dimension conjecture for all Artin algebras implies the Auslander--Reiten conjecture for all Artin algebras, see Wei \cite{Wei}. Our conclusion is more specific to the present problem. A counterexample determines the particular quotient $C=A/\Ann_A(T)$ on which the little finitistic dimension must be infinite. The two directions of our result can be summarized as
$$
\findim(C^{\rm op})<\infty
\quad\Longrightarrow\quad
T\text{ is }1\text{-tilting},
$$
for self-orthogonal $T$, while
$$
T\text{ is a counterexample}
\quad\Longrightarrow\quad
\findim(C^{\rm op})=\infty.
$$
The second implication is only a necessary condition for a counterexample. Its converse is not asserted. An algebra with infinite little finitistic dimension does not by itself produce a self-orthogonal $\tau$-tilting module with nonzero annihilator.

The first concrete application is the radical cube zero case.

\begin{theoremc}
Let $A$ be an Artin algebra with $\rad^3 A=0$. Then every self-orthogonal $\tau$-tilting right $A$-module is a classical $1$-tilting right $A$-module.
\end{theoremc}

Indeed, every quotient $C=A/\Ann_A(T)$ also has radical cube zero. Green and Zimmermann-Huisgen \cite{GreenHuisgen} proved that Artinian rings with vanishing radical cube have finite little finitistic dimension. Applying their theorem to $C^{\rm op}$ gives Theorem C. This extends the radical square zero result in \cite{ZhangII} to radical cube zero.

The paper is organized as follows. In Section 2, we recall the required notions from tilting theory and $\tau$-tilting theory. We also record two facts about minimal projective resolutions. In Section 3, we prove the finitistic dimension criterion, build a connection between the finitistic dimension conjecture and the self-orthogonal $\tau$-tilting conjecture, and explain the necessary condition for counterexamples. In Section 4, we prove the radical cube zero case and give further consequences of the main criterion. We also give examples which illustrate the positive result and the restriction on possible annihilators.

Throughout this paper, $A$ is an Artin $R$-algebra, where $R$ is a commutative Artin ring. Unless stated otherwise, all $A$-modules are finitely generated right $A$-modules. We write $\modA$ for the category of finitely generated right $A$-modules. The standard duality is denoted by $D$, and $\tau=D\operatorname{Tr}$ is the Auslander--Reiten translation.

\section{Preliminaries}

In this section, we recall the definitions and facts used later.

For a right $A$-module $M$, let $\add M$ be the full subcategory of direct summands of finite direct sums of copies of $M$. We write $\Fac M$ for the full subcategory of factor modules of objects in $\add M$. The annihilator of $M$ is
$$
\Ann_A(M)=\{a\in A\mid Ma=0\}.
$$
For a basic right $A$-module $M$, the number of isomorphism classes of indecomposable direct summands is denoted by $|M|$.

We use the standard notion of a classical $1$-tilting module. The equivalent forms below are recalled in \cite{HappelRingel}.

\begin{definition}\label{def:tilting}
A right $A$-module $T\in\modA$ is called a \emph{classical $1$-tilting module} if the following conditions hold.
\begin{enumerate}[\rm(i)]
\item $\pd_A T\leq1$.
\item $\Ext_A^1(T,T)=0$.
\item $|T|=|A|$.
\end{enumerate}
\end{definition}

The last condition can be replaced by the existence of an exact sequence
$$
0\longrightarrow A\longrightarrow T_0\longrightarrow T_1\longrightarrow0,
\quad T_0,T_1\in\add T.
$$
This standard equivalence is due to Bongartz \cite{Bongartz}.

The next definitions are due to Adachi, Iyama and Reiten \cite{AIR}.

\begin{definition}\label{def:tau-tilting}
Let $T\in\modA$. The right $A$-module $T$ is called $\tau$-rigid if
$$
\Hom_A(T,\tau T)=0.
$$
It is called $\tau$-tilting if it is $\tau$-rigid and $|T|=|A|$.
\end{definition}

The self-orthogonality condition used in the conjecture is the following one. See \cite{Zhang,Enomoto} for its role in this problem.

\begin{definition}\label{def:selforthogonal}
A right $A$-module $T\in\modA$ is called \emph{self-orthogonal} if
$$
\Ext_A^i(T,T)=0\quad(i\geq1).
$$
\end{definition}

The next fact explains why the annihilator is the natural object in the problem.

\begin{proposition}\label{prop:quotient-tilting}
Let $T$ be a $\tau$-tilting right $A$-module. Put
$$
I=\Ann_A(T),\quad C=A/I.
$$
Then $I\subseteq\rad A$, and $T$ is a classical $1$-tilting right $C$-module. Moreover, $T$ is a classical $1$-tilting right $A$-module if and only if $I=0$.
\end{proposition}

\begin{proof}
By \cite[Proposition~2.2(a)]{AIR}, every $\tau$-tilting module is sincere. More precisely, \cite[Proposition~2.5(3)]{CLZZ} shows that the annihilator of a $\tau$-tilting module is nilpotent. Hence $I\subseteq\rad A$. By \cite[Proposition~2.2(c)]{AIR}, the right $C$-module $T$ is a classical $1$-tilting module. Since $T$ is $\tau$-tilting, \cite[Proposition~2.2(b)]{AIR} implies that $T$ is a classical $1$-tilting right $A$-module exactly when it is faithful. Thus this happens exactly when $I=0$. The quotient viewpoint is also developed in \cite[Theorem~3.5]{CLZZ}.
\end{proof}

Since $I\subseteq\rad A$, the quotient satisfies
$$
\rad C=(\rad A)/I.
$$

We next recall the homological invariant used in the main result.

The little finitistic dimension originates in the classical work of Bass \cite{Bass}. We use the following convention.

\begin{definition}\label{def:findim}
Let $\Gamma$ be an Artin algebra. Its \emph{little finitistic dimension} is
$$
\findim(\Gamma)=\sup\{\pd_\Gamma X\mid X\in\mathsf{mod}\,\Gamma,\ \pd_\Gamma X<\infty\}.
$$
Thus $\findim(C^{\rm op})$ is computed on finitely generated left $C$-modules.
\end{definition}

We will also use two elementary facts about minimal projective resolutions. We use the standard properties of projective covers and radical morphisms for Artin algebras as in \cite{ARS}. Recall that a morphism between finitely generated projective modules is radical if it belongs to the radical of the category of finitely generated projective modules.

\begin{lemma}\label{lem:radical-epi}
Let $\Gamma$ be an Artin algebra. Let $f:P\to Q$ be a radical morphism between finitely generated projective right $\Gamma$-modules. If $Q\neq0$, then $f$ is not an epimorphism.
\end{lemma}

\begin{proof}
Suppose that $f$ is an epimorphism. Since $Q$ is projective, there is a morphism $s:Q\to P$ such that $fs=1_Q$. The radical is an ideal, so $fs$ belongs to $\rad\End_\Gamma(Q)$. This is impossible for a nonzero module $Q$.
\end{proof}

\begin{lemma}\label{lem:dual-radical}
Let $I\subseteq\rad A$ and put $C=A/I$. Let $u:P'\to P$ be a radical morphism between finitely generated projective right $A$-modules. Then $\Hom_A(P,C)$ and $\Hom_A(P',C)$ are finitely generated projective left $C$-modules. The induced map
$$
\Hom_A(u,C):\Hom_A(P,C)\longrightarrow\Hom_A(P',C)
$$
is radical. Moreover, its image is contained in the radical of $\Hom_A(P',C)$.
\end{lemma}

\begin{proof}
The first statement follows because $P$ and $P'$ are direct summands of finite direct sums of copies of $A$. We have $\Hom_A(A,C)\cong C$ as left $C$-modules.

For the second statement, decompose $P$ and $P'$ into indecomposable projective summands. It is enough to consider a component
$$
v:eA\longrightarrow fA,
$$
where $e$ and $f$ are primitive idempotents. Under the standard identification
$$
\Hom_A(eA,fA)\cong fAe,
$$
a radical morphism corresponds to an element of $f(\rad A)e$. Applying $\Hom_A(-,C)$ gives a map
$$
Cf\longrightarrow Ce
$$
whose defining element is the image of that element in $fCe$. Since
$$
\rad C=(\rad A)/I,
$$
this image belongs to $f(\rad C)e$. Hence each component of $\Hom_A(u,C)$ is radical. Therefore $\Hom_A(u,C)$ is a radical morphism between finitely generated projective left $C$-modules. The same component calculation gives
$\operatorname{Im}\Hom_A(u,C)\subseteq\rad\Hom_A(P',C).$
\end{proof}

\section{The finitistic dimension criterion}

In this section, we give the finitistic dimension criterion which will connect the self-orthogonal $\tau$-tilting conjecture with the finitistic dimension conjecture.

Let $T$ be a $\tau$-tilting right $A$-module. Put $I=\Ann_A(T)$ and $C=A/I$. By Proposition~\ref{prop:quotient-tilting}, the right $C$-module $T$ is classical $1$-tilting. Hence there is an exact sequence
\begin{equation}\label{eq:tilting-C}
0\longrightarrow C\longrightarrow T_0\longrightarrow T_1\longrightarrow0,
\qquad T_0,T_1\in\add T.
\end{equation}

The next result transfers a finite range of self-extension vanishing from $T$ to $C$.

\begin{proposition}\label{prop:vanishing-transfer}
Let $m\geq2$. Assume that
$$
\Ext_A^i(T,T)=0\quad(1\leq i\leq m).
$$
Then
$$
\Ext_A^j(C,C)=0\quad(1\leq j\leq m-1).
$$
In particular, if $T$ is self-orthogonal, then
$$
\Ext_A^j(C,C)=0\qquad(j\geq1).
$$
\end{proposition}

\begin{proof}
Apply $\Hom_A(-,T)$ to \eqref{eq:tilting-C}. Since $T_0,T_1\in\add T$, the long exact sequence gives
$$
\Ext_A^j(C,T)=0\qquad(1\leq j\leq m-1).
$$
The same vanishing holds when $T$ is replaced by $T_0$ or $T_1$.

Now apply $\Hom_A(C,-)$ to \eqref{eq:tilting-C}. For every $C$-module $M$, evaluation at $1+I$ gives a natural isomorphism
$$
\Hom_A(C,M)\cong M.
$$
Under these isomorphisms, the map
$$
\Hom_A(C,T_0)\longrightarrow\Hom_A(C,T_1)
$$
is the epimorphism $T_0\to T_1$ in \eqref{eq:tilting-C}. It follows that $\Ext_A^1(C,C)=0$.

For $2\leq j\leq m-1$, the long exact sequence contains
$$
\Ext_A^{j-1}(C,T_1)\longrightarrow\Ext_A^j(C,C)
\longrightarrow\Ext_A^j(C,T_0).
$$
The two outer terms vanish. Hence $\Ext_A^j(C,C)=0$.
\end{proof}

\begin{remark}\label{rem:finite-range}
The finite range in Proposition~\ref{prop:vanishing-transfer} is essential for the quantitative form of the main theorem. Full self-orthogonality is not needed once a bound for $\findim(C^{\rm op})$ is available. This distinguishes the result from the finite projective dimension criterion in \cite[Theorem~3.4]{Zhang}, where the hypothesis concerns $\pd_A T$ itself.
\end{remark}

We now prove the main result. Take a minimal projective resolution of $C$ as a right $A$-module
\begin{equation}\label{eq:minimal-resolution}
\cdots\longrightarrow P_2\xrightarrow{\partial_2}P_1
\xrightarrow{\partial_1}P_0=A\longrightarrow C\longrightarrow0.
\end{equation}
Since $I\subseteq\rad A$, the canonical epimorphism $A\to C$ is a projective cover. Therefore
$$
\operatorname{Im}\partial_1=I.
$$
For $i\geq0$, put
$$
Q_i=\Hom_A(P_i,C).
$$
Every $Q_i$ is a finitely generated projective left $C$-module. Applying $\Hom_A(-,C)$ to \eqref{eq:minimal-resolution} gives
\begin{equation}\label{eq:dual-complex}
0\longrightarrow Q_0\xrightarrow{\delta_0}Q_1
\xrightarrow{\delta_1}Q_2\xrightarrow{\delta_2}Q_3\longrightarrow\cdots.
\end{equation}
The $i$th cohomology of this complex is $\Ext_A^i(C,C)$.

We record two features of \eqref{eq:dual-complex}. First, $\delta_0=0$. Indeed, every $A$-homomorphism $A\to C$ vanishes on $I=\operatorname{Im}\partial_1$. Second, every $\delta_i$ with $i\geq1$ is radical by Lemma~\ref{lem:dual-radical}.

\begin{theorem}\label{thm:main}
Let $T$ be a $\tau$-tilting right $A$-module. Put
$$
I=\Ann_A(T),\quad C=A/I.
$$
Assume that
$$
d=\findim(C^{\rm op})<\infty.
$$
If
$$
\Ext_A^i(T,T)=0\quad(1\leq i\leq d+2),
$$
then $I=0$. Hence $T$ is a classical $1$-tilting right $A$-module.
\end{theorem}

\begin{proof}
Suppose that $I\neq0$. By Proposition~\ref{prop:vanishing-transfer},
$$
\Ext_A^i(C,C)=0\quad(1\leq i\leq d+1).
$$
Thus \eqref{eq:dual-complex} is exact at $Q_i$ for $1\leq i\leq d+1$.

We claim that
\begin{equation}\label{eq:Qi-nonzero}
Q_i\neq0\quad(1\leq i\leq d+2).
\end{equation}
Since $I\neq0$, the projective module $P_1$ is nonzero. Let $eA$ be a nonzero indecomposable projective summand of $P_1$. The idempotent $e$ does not belong to $I$ because $I\subseteq\rad A$. Hence
$$
\Hom_A(eA,C)\cong Ce\neq0.
$$
It follows that $Q_1\neq0$.

Assume that $Q_r\neq0$ for some $1\leq r\leq d+1$. If $Q_{r+1}=0$, exactness at $Q_r$ implies that the preceding map onto $Q_r$ is an epimorphism. For $r=1$, the preceding map is $\delta_0=0$, which is impossible. For $r\geq2$, the preceding map is radical. We apply Lemma~\ref{lem:radical-epi} to $C^{\rm op}$ and obtain a contradiction. This proves \eqref{eq:Qi-nonzero}.

Set
$$
X=\operatorname{Coker}(\delta_{d+1}:Q_{d+1}\to Q_{d+2}).
$$
By \eqref{eq:Qi-nonzero}, $Q_{d+2}\neq0$. The map $\delta_{d+1}$ is radical. Applying Lemma~\ref{lem:radical-epi} to $C^{\rm op}$ shows that it is not an epimorphism. Hence $X\neq0$. Exactness gives
\begin{equation}\label{eq:X-resolution}
0\longrightarrow Q_1\longrightarrow Q_2\longrightarrow\cdots
\longrightarrow Q_{d+1}\longrightarrow Q_{d+2}\longrightarrow X\longrightarrow0.
\end{equation}
Every differential between the projective terms is radical. The kernel of $Q_{d+2}\to X$ is contained in $\rad Q_{d+2}$. Hence this epimorphism is a projective cover. Therefore \eqref{eq:X-resolution} is the minimal projective resolution of $X$. Since $Q_1\neq0$, we obtain
$$
\pd_{{}_C}X=d+1.
$$
Here $X$ is a finitely generated left $C$-module. Equivalently, it is a finitely generated right $C^{\rm op}$-module. This contradicts $d=\findim(C^{\rm op})$.

Therefore $I=0$. The conclusion now follows from Proposition~\ref{prop:quotient-tilting}.
\end{proof}

\begin{corollary}\label{cor:selforthogonal}
Let $T$ be a self-orthogonal $\tau$-tilting right $A$-module and put $C=A/\Ann_A(T)$. If
$$
\findim(C^{\rm op})<\infty,
$$
then $T$ is a classical $1$-tilting right $A$-module.
\end{corollary}

\begin{proof}
This follows immediately from Theorem~\ref{thm:main}.
\end{proof}

The same argument gives a stronger obstruction when $T$ is self-orthogonal.

\begin{corollary}\label{cor:counterexample}
Let $T$ be a self-orthogonal $\tau$-tilting right $A$-module which is not a classical $1$-tilting module. Put
$$
C=A/\Ann_A(T).
$$
Then
$$
\findim(C^{\rm op})=\infty.
$$
For every $n\geq1$, there is a finitely generated left $C$-module $X_n$ with
$$
\pd_{{}_C}X_n=n.
$$
\end{corollary}

\begin{proof}
Since $T$ is self-orthogonal, Proposition~\ref{prop:vanishing-transfer} gives $\Ext_A^i(C,C)=0$ for all $i\geq1$. Since $T$ is not tilting, Proposition~\ref{prop:quotient-tilting} gives $I=\Ann_A(T)\neq0$. Thus the argument in the proof of Theorem~\ref{thm:main} works in every degree.

For $n\geq1$, set
$$
X_n=\operatorname{Coker}(Q_n\longrightarrow Q_{n+1}).
$$
The exact sequence
$$
0\longrightarrow Q_1\longrightarrow Q_2\longrightarrow\cdots
\longrightarrow Q_{n+1}\longrightarrow X_n\longrightarrow0
$$
is a minimal projective resolution. Hence $\pd_{{}_C}X_n=n$.
\end{proof}

\begin{corollary}\label{cor:FDC}
If the little finitistic dimension conjecture holds for all Artin algebras, then every self-orthogonal $\tau$-tilting right $A$-module is a classical $1$-tilting right $A$-module, that is, the self-orthogonal $\tau$-tilting conjecture holds for all Artin algebras.
\end{corollary}

\begin{proof}
For a self-orthogonal $\tau$-tilting right $A$-module $T$, the quotient $C=A/\Ann_A(T)$ is again an Artin algebra. The finitistic dimension conjecture gives $\findim(C^{\rm op})<\infty$. Corollary~\ref{cor:selforthogonal} applies.
\end{proof}

The global implication from the finitistic dimension conjecture to the Auslander--Reiten conjecture is known, see Wei \cite{Wei}. Corollary~\ref{cor:FDC} identifies a more specific obstruction in the present setting. A counterexample forces the canonical quotient $A/\Ann_A(T)$ itself to have infinite little finitistic dimension.

\begin{remark}\label{rem:comparison}
Corollaries~\ref{cor:selforthogonal} and \ref{cor:counterexample} give two complementary readings of the same argument. On the positive side, finite little finitistic dimension of the annihilator quotient forces a self-orthogonal $\tau$-tilting module to be faithful. On the negative side, a genuine counterexample forces the opposite quotient to have infinite little finitistic dimension. More is true in the negative direction because the proof constructs modules of every positive finite projective dimension.

This obstruction is different from the one in \cite{ZhangII}, which involves the endomorphism algebra $\End_A(T)^{\rm op}$. It is also different from the delooping criterion in \cite[Theorem~5.4]{CLZZ}. Our criterion is formulated entirely in terms of $A/\Ann_A(T)$. None of these conditions is claimed to imply the others in general.
\end{remark}

\section{Applications}

In this section, we give some applications of our main theorem. We first apply the main theorem to radical cube zero algebras.

\begin{theorem}\label{thm:radical-cube-zero}
Let $T$ be a self-orthogonal $\tau$-tilting right $A$-module and put
$$
C=A/\Ann_A(T).
$$
If
$$
\rad^3 C=0,
$$
then $T$ is a classical $1$-tilting right $A$-module.
\end{theorem}

\begin{proof}
The opposite algebra also satisfies $\rad^3(C^{\rm op})=0$. By Green and Zimmermann-Huisgen \cite{GreenHuisgen},
$$
\findim(C^{\rm op})<\infty.
$$
The conclusion follows from Corollary~\ref{cor:selforthogonal}.
\end{proof}

\begin{corollary}\label{cor:radical-cube-zero}
Let $A$ be an Artin algebra with $\rad^3 A=0$. Then every self-orthogonal $\tau$-tilting right $A$-module is a classical $1$-tilting right $A$-module, that is, the self-orthogonal $\tau$-tilting conjecture holds.
\end{corollary}

\begin{proof}
Let $I=\Ann_A(T)$ and $C=A/I$. Since $I\subseteq\rad A$, we have
$$
\rad C=(\rad A)/I.
$$
Hence $\rad^3 C=0$. The result follows from Theorem~\ref{thm:radical-cube-zero}.
\end{proof}

\begin{remark}\label{rem:radical-square}
Zhang \cite{ZhangII} proved the radical square zero case by annihilator separation. Theorem~\ref{thm:radical-cube-zero} extends that positive result to radical cube zero algebras. The proof is quite different. It uses the minimal $A$-projective resolution of the annihilator quotient and the theorem of Green and Zimmermann-Huisgen \cite{GreenHuisgen}.
\end{remark}

The next example illustrates Theorem~\ref{thm:radical-cube-zero} on an algebra which does not belong to the positive classes recalled in the introduction.

\begin{example}[A connected radical cube zero algebra]
\label{ex:cube-rep-infinite}
Let $k$ be an infinite field. Let $Q$ be the quiver with two vertices, one arrow $\alpha:1\to2$, and two loops $x,y$ at vertex $2$. Let $J$ be the arrow ideal of $kQ$. Put
$$
A=kQ/(J^3,xy-yx).
$$
Then $A$ is connected and nonlocal. Moreover,
$$
(\rad A)^3=0
\quad\text{and}\quad
(\rad A)^2\neq0.
$$
Hence Corollary~\ref{cor:radical-cube-zero} shows that every self-orthogonal $\tau$-tilting right $A$-module is a classical $1$-tilting right $A$-module.

(1) We first show that $A$ is representation-infinite.

Let $e_1$ be the idempotent at the first vertex. Then
$$
R=A/Ae_1A\cong k[x,y]/(x,y)^3.
$$
The quotient $R/(x,y)^2$ is $k[x,y]/(x,y)^2$. For each $\lambda\in k$, let $M_\lambda$ be the two-dimensional right $R/(x,y)^2$-module with basis $u_\lambda,v_\lambda$ and action
$$
u_\lambda x=v_\lambda,\quad
u_\lambda y=\lambda v_\lambda,\quad
v_\lambda x=v_\lambda y=0.
$$
Each $M_\lambda$ has simple top, and hence it is indecomposable. A direct comparison of the actions of $x$ and $y$ gives
$$
M_\lambda\simeq M_\mu
\quad\Longrightarrow\quad
\lambda=\mu.
$$
Thus $R$, and hence $A$, is representation-infinite. Since $R$ is a proper representation-infinite quotient of $A$, the algebra $A$ is not minimal representation-infinite.

(2) We next show that $A$ is not Iwanaga--Gorenstein.

The socle of the commutative local algebra $R$ is
$$
\operatorname{soc}R=(x,y)^2,
$$
which has dimension $3$ over $k$. By Bruns and Herzog \cite[Theorem~3.2.10]{BrunsHerzog}, a zero-dimensional commutative local ring is Gorenstein if and only if it has type one. Equivalently, its socle is one-dimensional over the residue field. Thus $R$ is not Gorenstein. Hence
$$
\operatorname{id}_R R=\infty.
$$
As a right $A$-module, $e_2A$ is naturally identified with $R$. Hence every finitely generated projective right $R$-module is projective as a right $A$-module. A projective $R$-resolution of a right $R$-module $X$, viewed by restriction of scalars along $A\to R$, is therefore also an $A$-projective resolution. It follows that
$$
\Ext_A^n(X,e_2A)\cong\Ext_R^n(X,R)
$$
for every finitely generated right $R$-module $X$ and every $n\geq0$. Since $A_A=e_1A\oplus e_2A$, we obtain
$$
\operatorname{id}_A A=\infty.
$$
Thus $A$ is not Iwanaga--Gorenstein.

In particular, this example is outside the local, syzygy-finite Gorenstein, representation-finite, minimal representation-infinite, and radical square zero cases recalled in the introduction.
\end{example}

The quotient formulation also recovers a familiar finite case in a slightly more flexible form.

\begin{corollary}\label{cor:representation-finite}
Let $T$ be a self-orthogonal $\tau$-tilting right $A$-module and put $C=A/\Ann_A(T)$. If $C$ is representation-finite, then $T$ is a classical $1$-tilting right $A$-module.
\end{corollary}

\begin{proof}
The standard duality for Artin algebras identifies finite representation type on the right and on the left \cite{ARS}. Hence there are only finitely many isomorphism classes of indecomposable finitely generated left $C$-modules. Their finite projective dimensions are bounded. Thus $\findim(C^{\rm op})<\infty$, and Corollary~\ref{cor:selforthogonal} applies.
\end{proof}

\begin{remark}\label{rem:repfinite}
When $A$ itself is representation-finite, Corollary~\ref{cor:representation-finite} recovers the known positive answer from \cite[Proposition~5.3]{CLZZ}. The quotient formulation is slightly more flexible. It only asks that the particular algebra $A/\Ann_A(T)$ attached to $T$ be representation-finite.
\end{remark}

The finite version of Theorem~\ref{thm:main} is useful when an explicit bound is available.

\begin{proposition}\label{prop:finite-test}
Let $T$ be a $\tau$-tilting right $A$-module and put $C=A/\Ann_A(T)$. Assume that
$$
\findim(C^{\rm op})\leq d.
$$
If
$$
\Ext_A^i(T,T)=0\quad(1\leq i\leq d+2),
$$
then $T$ is a classical $1$-tilting right $A$-module.
\end{proposition}

\begin{proof}
Apply Theorem~\ref{thm:main} with $\findim(C^{\rm op})\leq d$. If the actual finitistic dimension is smaller than $d$, the assumed vanishing contains the range needed in Theorem~\ref{thm:main}.
\end{proof}

\begin{corollary}
\label{cor:finite-range-test}
Let $T$ be a $\tau$-tilting right $A$-module and put $C=A/\Ann_A(T)$.
\begin{enumerate}[\rm(i)]
\item If $C$ is hereditary and
$$
\Ext_A^i(T,T)=0\quad(1\leq i\leq3),
$$
then $T$ is a classical $1$-tilting right $A$-module.
\item If $C$ is semisimple and
$$
\Ext_A^1(T,T)=0=\Ext_A^2(T,T),
$$
then $T$ is a classical $1$-tilting right $A$-module.
\end{enumerate}
\end{corollary}

\begin{proof}
If $C$ is hereditary, then $\findim(C^{\rm op})\leq1$. If $C$ is semisimple, then $\findim(C^{\rm op})=0$. Both assertions follow from Proposition~\ref{prop:finite-test}.
\end{proof}

This finite form shows that full self-orthogonality is not required when a small finitistic dimension bound is known.

The quantitative form also excludes prescribed nonzero annihilators. We use the same algebra as in Example~\ref{ex:cube-rep-infinite}.

\begin{example}[A prescribed annihilator is excluded]
\label{ex:excluded-annihilator}
Let $A$ be the algebra in Example~\ref{ex:cube-rep-infinite} and put
$$
I=\rad A.
$$
Since $A$ has two simple modules,
$$
C=A/I\cong k\times k
$$
is semisimple. Hence $\findim(C^{\rm op})=0$.

Suppose that a $\tau$-tilting right $A$-module $T$ satisfies
$$
\Ann_A(T)=I
$$
and
$$
\Ext_A^1(T,T)=0=\Ext_A^2(T,T).
$$
Theorem~\ref{thm:main} with $d=0$ gives $\Ann_A(T)=0$, contrary to $I\neq0$. Therefore no such $T$ exists. This illustrates the finite form of our criterion. On this algebra, the first two self-extension groups already exclude $\rad A$ as the annihilator of a $\tau$-tilting module.
\end{example}

\begin{remark}\label{rem:final}
The main theorem narrows the search for a counterexample. A nonzero annihilator alone is not enough. The corresponding quotient must have infinite little finitistic dimension. In fact, it must contain finitely generated left modules of every positive finite projective dimension. This is only a necessary condition. The existence of such a quotient does not imply the existence of a self-orthogonal $\tau$-tilting module with nonzero annihilator.
\end{remark}

There is no known Artin algebra which can serve as a concrete example for the conclusion $\findim(C^{\rm op})=\infty$ in Corollary~\ref{cor:counterexample}. The finitistic dimension conjecture remains open for Artin algebras. See, for example, \cite{Giatagantzidis}. Thus the negative part of our theorem should be read as a reduction. Any counterexample to the self-orthogonal $\tau$-tilting conjecture would at the same time give a counterexample to the finitistic dimension conjecture.

\vspace{6mm}

\textbf{Acknowledgements.} Xiaojin Zhang is supported by the National Natural Science Foundation of China (Grant Nos. 12171207 and 12371038). Panyue Zhou is supported by the National Natural Science Foundation of China (Grant No. ~12371034).

\vspace{4mm}

\hspace{-4mm}\textbf{Data Availability}\hspace{2mm} Data sharing not applicable to this article as no datasets were generated or analysed during
the current study.
\vspace{2mm}

\hspace{-4mm}\textbf{Conflict of Interests}\hspace{2mm} The authors declare that they have no conflicts of interest to this work.


\vspace{4mm}


\begin{thebibliography}{CLZZ}

\bibitem[AIR]{AIR}
T.~Adachi, O.~Iyama and I.~Reiten,
\emph{$\tau$-tilting theory},
Compos. Math. \textbf{150} (2014), no.~3, 415--452.

\bibitem[ARS]{ARS}
M.~Auslander, I.~Reiten and S.~O.~Smal\o,
\emph{Representation Theory of Artin Algebras},
Cambridge Studies in Advanced Mathematics, Vol.~36, Cambridge University Press, Cambridge, 1995.

\bibitem[AuR]{AuslanderReiten}
M.~Auslander and I.~Reiten,
\emph{On a generalized version of the Nakayama conjecture},
Proc. Amer. Math. Soc. \textbf{52} (1975), 69--74.

\bibitem[B]{Bass}
H.~Bass,
\emph{Finitistic dimension and a homological generalization of semi-primary rings},
Trans. Amer. Math. Soc. \textbf{95} (1960), 466--488.

\bibitem[Bo]{Bongartz}
K.~Bongartz,
\emph{Tilted algebras},
in \emph{Representations of Algebras}, Lecture Notes in Mathematics, Vol.~903,
Springer-Verlag, Berlin--New York, 1981, 26--38.

\bibitem[BH]{BrunsHerzog}
W.~Bruns and J.~Herzog,
\emph{Cohen--Macaulay Rings},
revised edition, Cambridge Studies in Advanced Mathematics, Vol.~39,
Cambridge University Press, Cambridge, 1998.

\bibitem[CT]{ChangTang}
W.~Chang and Q.~Tang,
\emph{Tilting completion and self-orthogonality modules},
arXiv:2608.20722, 2026.

\bibitem[CLZZ]{CLZZ}
X.-W.~Chen, Z.-W.~Li, X.~Zhang and Z.~Zhao,
\emph{Comparing $\tau$-tilting modules and $1$-tilting modules},
arXiv:2501.02466, 2025.

\bibitem[E]{Enomoto}
H.~Enomoto,
\emph{Maximal self-orthogonal modules and a new generalization of tilting modules},
arXiv:2301.13498, 2023.

\bibitem[GR]{GeissReiten}
C.~Geiss and I.~Reiten,
\emph{Gentle algebras are Gorenstein},
in \emph{Representations of algebras and related topics},
Fields Inst. Commun. \textbf{45}, Amer. Math. Soc., Providence, RI, 2005, 129--133.


\bibitem[Gi]{Giatagantzidis}
O.~Giatagantzidis,
\emph{Radical preservation and the finitistic dimension},
Bull. London Math. Soc. \textbf{58} (2026), no.~1, Paper No.~e70222, 23 pp.


\bibitem[GZH]{GreenHuisgen}
E.~L.~Green and B.~Zimmermann-Huisgen,
\emph{Finitistic dimension of artinian rings with vanishing radical cube},
Math. Z. \textbf{206} (1991), 505--526.

\bibitem[HR]{HappelRingel}
D.~Happel and C.~M.~Ringel,
\emph{Tilted algebras},
Trans. Amer. Math. Soc. \textbf{274} (1982), no.~2, 399--443.

\bibitem[M]{Marczinzik}
R.~Marczinzik,
\emph{On self-orthogonal modules in Iwanaga--Gorenstein rings},
Ann. Represent. Theory \textbf{1} (2024), no.~1, 67--70.

\bibitem[T]{Tachikawa}
H.~Tachikawa,
\emph{Quasi-Frobenius Rings and Generalizations. QF-3 and QF-1 Rings},
Lecture Notes in Mathematics, Vol.~351, Springer-Verlag, Berlin--New York, 1973.

\bibitem[W]{Wei}
J.~Wei,
\emph{Tilting invariance of the Auslander--Reiten conjecture},
Math. Res. Lett. \textbf{17} (2010), no.~1, 171--176.

\bibitem[Z1]{Zhang}
X.~Zhang,
\emph{Self-orthogonal $\tau$-tilting modules and tilting modules},
J. Pure Appl. Algebra \textbf{226} (2022), no.~3, Paper No.~106860.

\bibitem[Z2]{ZhangII}
X.~Zhang,
\emph{Self-Orthogonal $\tau$-Tilting Modules and Tilting Modules II: Annihilator Separation},
arXiv:2608.27937, 2026.

\end{thebibliography}
\end{document}